\documentclass[a4paper]{amsart}
\usepackage{amssymb,amsthm,amsmath}
\usepackage{graphicx}
\usepackage{geometry}
\usepackage{mathtools}
\usepackage{stmaryrd}
\usepackage{slashbox}

\usepackage{multicol}

\def\Z{\mathbb{Z}}

\def\mQ{\mathcal{Q}}
\def\rev{\operatorname{rev}}
\def\Atilde{\tilde{A}}
\def\Apoint{A^\bullet}

\newtheorem{thm}{Theorem}[section]
\newtheorem*{thm*}{Theorem}
\newtheorem{lem}[thm]{Lemma}
\newtheorem{cor}[thm]{Corollary}
\newtheorem*{cor*}{Corollary}

\newtheorem{eg}[thm]{Example}

\theoremstyle{definition}
\newtheorem{Def}[thm]{Definition}

\theoremstyle{remark}
\newtheorem{rmk}[thm]{Remark}

\newcommand{\Dfn}[1]{\emph{#1}} 

\title{Counting the number of elements in the mutation classes of $\Atilde_n-$quivers}

\author{Janine Bastian}
\address{Institut f\"ur Algebra, Zahlentheorie und Diskrete Mathematik, Leibniz
  Universit\"at Hannover, Welfengarten 1, D-30167 Hannover, Germany}
\email{bastian@math.uni-hannover.de}
\urladdr{http://www.iazd.uni-hannover.de/~bastian/}
\author{Thomas Prellberg}
\address{School of Mathematical Sciences, Queen Mary University of London,
  Mile End Road, London E1 4NS, United Kingdom}
\email{t.prellberg@qmul.ac.uk}
\urladdr{http://maths.qmul.ac.uk/~tp/}
\author{Martin Rubey}
\address{Institut f\"ur Algebra, Zahlentheorie und Diskrete Mathematik, Leibniz
  Universit\"at Hannover, Welfengarten 1, D-30167 Hannover, Germany}
\email{martin.rubey@math.uni-hannover.de}
\urladdr{http://www.iazd.uni-hannover.de/~rubey/}
\author{Christian Stump}
\address{LaCIM, Universit\'e du Qu\'ebec \`a Montr\'eal, 201,
Pr\'esident-Kennedy, 4th floor, Montr\'eal (Qu\'ebec) H2X 3Y7 ,
Canada}
\email{christian.stump@lacim.ca}
\urladdr{http://homepage.univie.ac.at/christian.stump/}

\begin{document}
\begin{abstract}
  In this article we prove explicit formulae for the number of
  non-isomorphic cluster-tilted algebras of type $\Atilde_n$ in the
  derived equivalence classes.  In particular, we obtain the number
  of elements in the mutation classes of quivers of type $\Atilde_n$.
  As a by-product, this provides an alternative proof for the number
  of quivers mutation equivalent to a quiver of Dynkin type $D_n$
  which was first determined by Buan and Torkildsen in
  \cite{BuanTorkildsen2009}.
\end{abstract}
\maketitle


\section{Introduction}

Quiver mutation is a central element in the recent theory of cluster algebras
introduced by Fomin and Zelevinsky in \cite{FominZelevinsky2002}.
It is an elementary operation on quivers which generates an equivalence
relation.  The mutation class of a quiver $Q$ is the class of all quivers which
are mutation equivalent to $Q$.

The mutation class of quivers of type $A_n$ is the class containing
all quivers mutation equivalent to a quiver whose underlying graph is
the Dynkin diagram of type $A_n$, shown in
Figure~\ref{fig:Dynkin}$(a)$.  This mutation class was described by
Caldero, Chapoton and Schiffler~\cite{MR2187656}
in terms of triangulations.  An explicit characterisation of the
quivers themselves can be found in Buan and Vatne
in~\cite{BuanVatne2008}.  The corresponding task for type $D_n$,
shown in Figure~\ref{fig:Dynkin}$(b)$, was accomplished by Vatne
in~\cite{Vatne2008}.  Furthermore, an explicit formula for the number
of quivers in the mutation class of type $A_n$ was given by
Torkildsen in~\cite{Torkildsen2008} and of type $D_n$ by Buan and
Torkildsen in \cite{BuanTorkildsen2009}.

In this article, we consider quivers of type $\Atilde_{n-1}$. That is, all
quivers mutation equivalent to a quiver whose underlying graph is the extended
Dynkin diagram of type $\Atilde_{n-1}$, i.e., the $n$-cycle, see
Figure~\ref{fig:Dynkin}$(c)$.  If this cycle is oriented, then we get the
mutation class of $D_n$, see Fomin et al.\ in \cite{FominShapiroThurston2008}
and Type IV in \cite{Vatne2008}.
If the cycle is non-oriented, we get the mutation classes of $\Atilde_{n-1}$,
studied by the first named author in \cite{Bastian2009}.  The purpose of this
paper is to give an explicit formula for the number of quivers in the mutation
classes of quivers of type $\Atilde_{n-1}$.

A cluster-tilted algebra $C$ of type $\Atilde_{n-1}$ is finite
dimensional over an algebraically closed field $K$.  Therefore, there
exists a quiver $Q$ which is in one of the mutation classes of
$\Atilde_{n-1}$ (see for instance Buan, Marsh and
Reiten~\cite{BuanMarshReiten2008} or Assem et
al.~\cite{AssemBruestleCharbonneauJodoinPlamondon2009}) and an
admissible ideal I of the path algebra $KQ$ of $Q$ such that $C \cong
KQ/I$.  Furthermore, two cluster-tilted algebras of the same type are
isomorphic if and only if the corresponding quivers are isomorphic as
directed graphs.

Thus, we also obtain the number of non-isomorphic cluster-tilted
algebras of type $\Atilde_{n-1}$.  In fact, we prove a more refined
counting theorem.  Namely, one can classify these algebras up to
derived equivalence, see \cite{Bastian2009}.  Each equivalence class
is determined by four parameters, $r_1$, $r_2$, $s_1$ and $s_2$,
where $r_1+2r_2+s_1+2s_2=n$, up to interchanging $r_1$, $r_2$ and
$s_1$, $s_2$.  Without loss of generality, we can therefore assume
that $r_1<s_1$ or $r_1=s_1$ and $r_2\leq s_2$.  Given positive
integers $r$ and $s$ with $r+s=n$, the set of equivalence classes
with $r_1+2r_2=r$ and $s_1+2s_2=s$ corresponds to one mutation class
of quivers.

\begin{thm*}
  The number of cluster-tilted algebras in the derived equivalence
  classes with parameters $r_1, r_2, s_1$ and $s_2$ is given by
  \begin{multline*}
    \sum_{k|r,k|r_2,k|s,k|s_2}\frac{\phi(k)}k(-1)^{(r+r_2+s+s_2)/k}\\
    \sum_{\substack{i,j\geq0\\
        (i,j)\neq(0,0)}}\frac{(-1)^{i+j}}{2(i+j)}\binom{2i}{i,2i-r/k,r_2/k,(r-r_2)/k-i}
    \binom{2j}{j,2j-s/k,s_2/k,(s-s_2)/k-j}
  \end{multline*}
  if $r_1<s_1$ or $r_1=s_1$ and $r_2<s_2$.  Otherwise, if $r_1=s_1$
  and $r_2=s_2$, the number is
  \begin{multline*}
    2^{r-2r_2-2}\binom{r}{r_2,r_2,r-2r_2}\\ %
    +\sum_{\substack{k|r,k|r_2\\i,j\geq0\\(i,j)\neq(0,0)}}\frac{\phi(k)}k
    \frac{(-1)^{i+j}}{4(i+j)}
    \binom{2i}{i,2i-r/k,r_2/k,(r-r_2)/k-i}\binom{2j}{j,2j-r/k,r_2/k,(r-r_2)/k-j}.
  \end{multline*}
  Here $\phi(k)$ is Euler's totient function, i.e., the number of
  $1\leq d<k$ coprime to $k$ and $\binom{m}{m_1,m_2,\dots,m_\ell}$
  with $m_1+m_2+\dots+m_\ell=m$ denotes the multinomial coefficient.

  In particular, for $r = r_1 + 2r_2$ and $s = s_1 +2s_2$, we obtain
  the number $\tilde{a}(r,s)$ of quivers mutation equivalent to a
  non-oriented $n$-cycle with $r$ arrows oriented in one direction
  and $s$ arrows oriented in the other direction:
  $$
  \tilde{a}(r,s) =
  \begin{cases}
    \frac{1}{2} \sum\limits_{k|r,k|s}\frac{\phi(k)}{r+s}%
    \binom{2r/k}{r/k}\binom{2s/k}{s/k}& \text{if } r< s,\\[12pt]
    \frac{1}{2} \left(\frac{1}{2} \binom{2r}{r}%
      + \sum\limits_{k|r} \frac{\phi(k)}{4r}%
      \binom{2r/k}{r/k}^2 \right) & \text{if } r=s.
  \end{cases}
  $$
\end{thm*}

Additionally, we obtain the number of quivers in the mutation class
of a quiver of Dynkin type $D_n$. This formula was first determined in
\cite{BuanTorkildsen2009}:

\begin{cor*} The number of quivers of type $D_n$, for $n \geq 5$, is given by
  $$\tilde{a}(0,n)=\sum_{d|n} \frac{\phi(n/d)}{2n} \binom{2d}{d}.$$
  The number of quivers of type $D_4$ is $6$.
\end{cor*}

The paper is organized as follows.  In Section~\ref{sec:pres} we
collect some basic notions about quiver mutation.  Furthermore, we
present the classification of quivers of type $\Atilde_{n-1}$
according to the parameters $r$ and $s$ mentioned above, as given in
\cite{Bastian2009}.  In Section~\ref{sec:comb-gramm} we restate the
classification as a combinatorial grammar.  Using \lq
generatingfunctionology\rq\ we obtain the formulae for the
assymmetric case where $r_1\neq s_1$ or $r_2\neq s_2$.  For the case
$r_1=s_1$ and $r_2=s_2$ some additional combinatorial considerations,
counting the number of quivers invariant under reflection, yield the
result stated above.

The formulae for $\tilde{a}(r,s)$ are developed in parallel.  In
fact, it is remarkable that the generating function including
variables for the parameters $r_2$ and $s_2$ can be obtained by
\emph{specialising} the much simpler generating function having
variables for the parameters $r$ and $s$ only.  Moreover, extracting
the coefficient of $p^r q^s x^{r_2} y^{s_2}$ in a naive way from the
equations obtained from combinatorial grammars results in a much uglier
five-fold sum, instead of the three-fold sum stated in the main
theorem.

Finally, at the end of Section~\ref{sec:comb-gramm} we prove the
formula for the number of quivers in the mutation class of type
$D_n$, by exhibiting an appropriate bijection between these and a
subclass of the objects counted in Section~\ref{sec:number-quivers}.

\medskip

{\bf Acknowledgements:} We would like to thank Thorsten Holm for
invaluable comments on a preliminary version of this article, and Ira
Gessel for providing the beautiful proof of Lemma~\ref{lem:Gessel}.
We also would like to thank Christian Krattenthaler who gave an
elementary proof of the same lemma, of which at first we only had a
computer assisted proof.

\begin{figure}
	\begin{tabular}{ccc}
		\begin{tabular}{c}
		\includegraphics{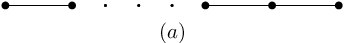}\\
		\includegraphics{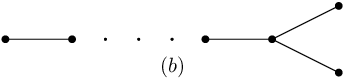}
		\end{tabular}& \hspace{7pt} &
		\begin{tabular}{c} \\ \includegraphics{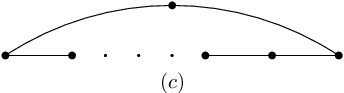} \\ \\ \end{tabular}
	\end{tabular}
	\caption{The Dynkin diagrams of types $A_n$ and $D_n$ and the extended
          Dynkin diagram of type $\Atilde_{n-1}$, assuming that all diagrams have $n$ vertices.}
    \label{fig:Dynkin}
\end{figure}


\section{Preliminaries} \label{sec:pres}

A \Dfn{quiver} $Q$ is a (finite) directed graph where loops and multiple arrows are
allowed. Formally, $Q$ is a quadruple $Q = (Q_0, Q_1, h, t)$ consisting of two
finite sets $Q_0,Q_1$ whose elements are called \Dfn{vertices} and \Dfn{arrows}
resp., and two functions
$$
h : Q_1 \rightarrow Q_0, \quad t : Q_1 \rightarrow Q_0,
$$
assigning a \Dfn{head} $h(\alpha)$ and a \Dfn{tail} $t(\alpha)$ to each arrow
$\alpha \in Q_1$.

Moreover, if $t(\alpha) = i$ and $h(\alpha) = j$ for $i,j \in Q_0$, we say
$\alpha$ is an arrow from $i$ to $j$ and write $i \xrightarrow{\alpha} j$. In
this case, $i$ and $\alpha$ as well as $j$ and $\alpha$ are called
\Dfn{incident} to each other. As usual, two quivers are considered to be equal
if they are isomorphic as directed graphs. The \Dfn{underlying graph} of a
quiver $Q$ is the graph obtained from $Q$ by replacing the arrows in $Q$
by undirected edges.

A quiver $Q^\prime =(Q'_0,Q'_1,h',t')$ is a \Dfn{subquiver} of a quiver $Q =
(Q_0,Q_1,h,t)$ if $Q^\prime_0 \subseteq Q_0$ and $Q^\prime_1 \subseteq Q_1$ and
where $h'(\alpha) = h(\alpha) \in Q'_0, \ t'(\alpha) = t(\alpha) \in Q'_0$ for
any arrow $\alpha \in Q'_1$.  A subquiver is called a \Dfn{full subquiver}
if for any two vertices $i$ and $j$ in the subquiver, the subquiver also
will contain all arrows between $i$ and $j$ present in $Q$.

An \Dfn{oriented cycle} is a subquiver of a quiver whose underlying
graph is a cycle on at least two vertices and whose arrows are all
oriented in the same direction, i.e., every vertex has outdegree $1$.
By contrast, a \Dfn{non-oriented cycle} is a subquiver of a quiver
whose underlying graph is a cycle, but not all of its arrows are
oriented in the same direction.

Throughout the paper, unless explicitly stated, we assume that
\begin{itemize}
\item quivers do not have loops or oriented $2$-cycles, i.e., $h(\alpha) \neq
  t(\alpha)$ for any arrow $\alpha$ and there do not exist arrows
  $\alpha,\beta$ such that $h(\alpha) = t(\beta)$ and $h(\beta) = t(\alpha)$;
\item quivers are connected.
\end{itemize}

\subsection{Quiver mutation}

In \cite{FominZelevinsky2002}, Fomin and Zelevinsky introduced
the \Dfn{quiver mutation} of a quiver $Q$ without loops and oriented $2$-cycles
at a given vertex of $Q$:
\begin{Def}
  Let $Q$ be a quiver. The \Dfn{mutation} of $Q$ at a vertex $k$ is defined to
  be the quiver $Q^* := \mu_k(Q)$ given as follows.
  \begin{enumerate}
  \item Add a new vertex $k^*$.
  \item Suppose that the number of arrows $i \rightarrow k$ in $Q$ equals $a$,
    the number of arrows $k \rightarrow j$ equals $b$ and the number of arrows
    $j \rightarrow i$ equals $c \in \Z$.  Then we have $c-ab$ arrows $j
    \rightarrow i$ in $Q^*$.  Here, a negative number of arrows means arrows in
    the opposite direction.
  \item For any arrow $i \rightarrow k$ (resp.\ $k \rightarrow j$) in $Q$ add an
    arrow $k^* \rightarrow i$ (resp.\ $j \rightarrow k^*$) in $Q^*$.
  \item Remove the vertex $k$ and all its incident arrows.
  \end{enumerate}
  No other arrows are affected by this operation.
\end{Def}

Note that mutation at sinks or sources only means changing the
direction of all incoming and outgoing arrows. Mutation at a vertex
$k$ is an involution on quivers, that is, $\mu_k(\mu_k(Q)) = Q$. It
follows that mutation generates an equivalence relation and we call
two quivers \Dfn{mutation equivalent} if they can be obtained from
each other by a finite sequence of mutations. The \Dfn{mutation
  class} of a quiver $Q$ is the class of all quivers (up to
relabelling of the vertices) which are mutation equivalent to $Q$.

We have the following well-known lemma:
\begin{lem}
  If quivers $Q, Q'$ have the same underlying graph which is a tree, then $Q$
  and $Q'$ are mutation equivalent.
\end{lem}

This lemma implies that one can speak of quivers associated to a
simply-laced Dynkin diagram, i.e., the Dynkin diagram of type $A_n,
D_n$ or $E_n$: we define a \Dfn{quiver of type} $A_n$ (resp.\ $D_n,
E_n$) to be a quiver in the mutation class of all quivers whose
underlying graph is the Dynkin diagram of type $A_n$ (resp.\ $D_n,
E_n$).  We remark that some authors use this term to refer to an
orientation of the Dynkin diagram of type $A_n$ (resp.\ $D_n, E_n$).

One can easily check that the oriented $n$-cycle is also of type
$D_n$, as has been done in \cite[Type IV]{Vatne2008}.  Two
non-oriented $n$-cycles are mutation equivalent if and only if the
number of arrows oriented clockwise coincide, or the number of arrows
oriented clockwise in one cycle agrees with the number of arrows
oriented anti-clockwise in the other cycle. This was shown in
\cite{Bastian2009, FominShapiroThurston2008} and is restated in
Theorem~\ref{thm:description_mutation_classes}.  We call quivers in
those mutation classes \Dfn{quivers of type} $\Atilde_{n-1}$.  They
will be described in more detail in Section~\ref{sec:mutation classes
  of Atilde quivers}.  In Figure~\ref{fig:Dynkin}, the Dynkin
diagrams of types $A_n$ and $D_n$ and the extended Dynkin diagram of
$\Atilde_{n-1}$ are shown.
		
\begin{eg}
	The mutation class of type $\Atilde_3$ of the non-oriented cycles with two arrows in each direction is given by
	
  \begin{center}
  	{\scriptsize
    	\begin{tabular}{c} \includegraphics[scale=0.8]{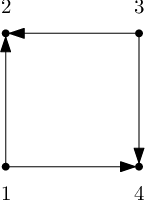} 	\end{tabular} $\xleftrightarrow{\mu_3}$
			\begin{tabular}{c} \includegraphics[scale=0.8]{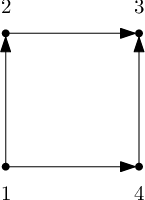} 	\end{tabular} $\xleftrightarrow{\mu_4}$ 
  		\begin{tabular}{c} \includegraphics[scale=0.8]{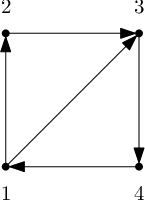} 	\end{tabular} $\xleftrightarrow{\mu_2}$
  		\begin{tabular}{c} \includegraphics[scale=0.8]{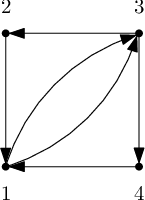}	\end{tabular}.
  	}
  \end{center}
  
  The mutation class of type $\Atilde_3$ of the non-oriented cycle with $3$ arrows in one direction and $1$ arrow in the other is given by
  
	\begin{center}
		{\scriptsize
    	\begin{tabular}{c} \includegraphics[scale=0.8]{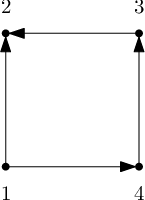} 	\end{tabular} $\xleftrightarrow{\mu_4}$ 
    	\begin{tabular}{c} \includegraphics[scale=0.8]{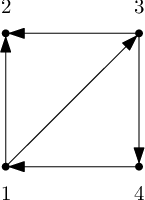} 	\end{tabular} $\xleftrightarrow{\mu_2}$  
    	\begin{tabular}{c} \includegraphics[scale=0.8]{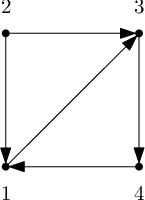} 	\end{tabular} $\xleftrightarrow{\mu_1}$
      \begin{tabular}{c} \includegraphics[scale=0.8]{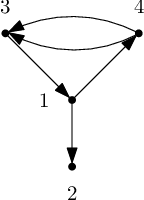} 	\end{tabular} $\xleftrightarrow{\mu_2}$ 
      \begin{tabular}{c} \includegraphics[scale=0.8]{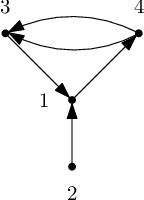}		\end{tabular}.
		}
	\end{center}
\end{eg}

\subsection{Mutation classes of $\Atilde_{n-1}-$quivers}
\label{sec:mutation classes of Atilde quivers}

Following \cite{Bastian2009}, we now describe the mutation classes of quivers
of type $\Atilde_{n-1}$ in more detail:

\begin{Def} \label{dfn:description_mutation_class}
Let $\mQ_{n-1}$ be the class of quivers with $n$ vertices which satisfy the following
conditions:
\begin{enumerate}
\item There exists precisely one full subquiver which is a non-oriented cycle of
length $\geq 2$. Thus, if the length is two, it is a double arrow.
\item For each arrow $x \xrightarrow{\alpha} y$ in this non-oriented cycle, there may
(or may not) be a vertex $z_{\alpha}$ which is not on the non-oriented cycle, such that
there is an oriented $3$-cycle of the form
\begin{center}
\includegraphics[scale=0.9]{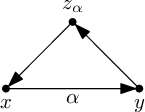}
\end{center}
Apart from the arrows of these oriented $3$-cycles there are no other arrows
incident to vertices on the non-oriented cycle.
\item If we remove all vertices in the non-oriented cycle and their incident
arrows, the result is a disconnected union of quivers, one for each
$z_{\alpha}$.
These are quivers of type $A_{k_\alpha}$ for ${k_\alpha} \geq 1$
(see \cite{BuanVatne2008} for the mutation class of $A_n$), and the
vertices $z_{\alpha}$ have at most two incident arrows in these quivers.
Furthermore, if a vertex $z_{\alpha}$ has two incident arrows in such a quiver,
then $z_{\alpha}$ is a vertex in an oriented $3$-cycle.
We call these quivers \Dfn{rooted quivers of type $A$} with root $z_{\alpha}$.
Note that this is a similar description as for Type IV in~\cite{Vatne2008}.

The rooted quiver of type $A$ with root $z_\alpha$ is called \Dfn{attached}
to the arrow~$\alpha$.
\end{enumerate}
\end{Def}

\begin{rmk}
Our convention is to choose only one of the double arrows to be part of the
oriented $3$-cycle in the following case:
\begin{center}
\includegraphics[scale=0.9]{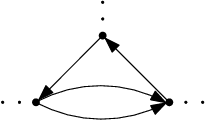}
\end{center}
\end{rmk}

\begin{eg}
  The following quiver is of type $\Atilde_{21}$:
  \begin{center}
    \includegraphics[scale=0.8]{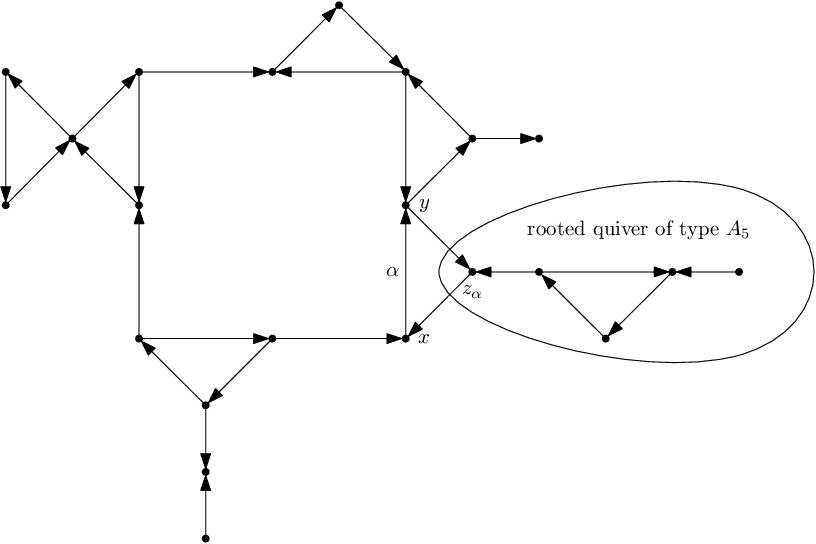}
  \end{center}
\end{eg}

\begin{Def}\label{dfn:realization}
	A \Dfn{realization} of a quiver $Q \in \mQ_{n-1}$ is the quiver
	together with an embedding of the non-oriented cycle into the plane.
	We do not care about a particular embedding of the other arrows,
	i.e., there are at most two different realizations of any given quiver.
    Thus, we can speak of clockwise and anti-clockwise oriented arrows
	in the non-oriented cycle.
\end{Def}

We will see in Section~\ref{sec:comb-gramm} that it is straightforward to count
the number of possible realizations of quivers in a mutation class of
$\Atilde_{n-1}$.  Since the two realizations of a quiver may coincide, we will need
an additional argument to count the number of quivers themselves.

As in \cite{Bastian2009} we can define parameters $r_1, \ r_2, \ s_1$ and $s_2$
for a realization of a quiver $Q \in \mQ_{n-1}$ as follows:

\begin{Def}\label{dfn:parameters}
  Let $Q$ be a quiver in $\mQ_{n-1}$ and fix a realization of $Q$.  The arrows in
  $Q$ which are part of the non-oriented cycle are called \Dfn{base arrows}.
  Let $r_1$ be the number of arrows which are not part of any oriented
  $3$-cycle and which are either
  \begin{enumerate}
  \item base arrows and oriented anti-clockwise, or
  \item contained in a rooted quiver of type $A$ attached to a base arrow
    $\alpha$ which is oriented anti-clockwise.
  \end{enumerate}

  \begin{minipage}{0.95\linewidth}
    \hfill
    \begin{minipage}{0.05\linewidth}
      (1)
    \end{minipage}
    \hfill
    \begin{minipage}{0.4\linewidth}
      \begin{center}
        \includegraphics[scale=0.8]{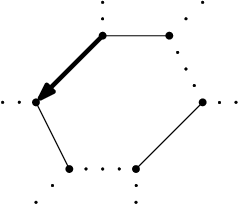}
      \end{center}
    \end{minipage}
    \hfill
    \begin{minipage}{0.05\linewidth}
      (2)
    \end{minipage}
    \hfill
    \begin{minipage}{0.4\linewidth}
      \begin{center}
        \includegraphics[scale=0.8]{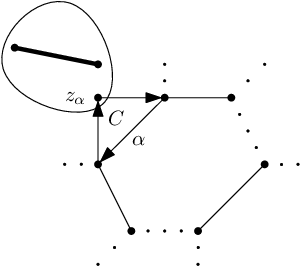}
      \end{center}
    \end{minipage}
  \end{minipage}

  \vspace{5pt}

  Let $r_2$ be the number of oriented $3$-cycles
  \begin{enumerate}
  \item which share an arrow $\alpha$ with the non-oriented cycle
  and $\alpha$ (a base arrow) is oriented anti-clockwise, or
  \item which are contained in a rooted quiver of type $A$ attached to a base arrow
    $\alpha$ which is oriented anti-clockwise.
  \end{enumerate}

  \vspace{5pt}

  \begin{minipage}{0.95\linewidth}
    \hfill
    \begin{minipage}{0.05\linewidth}
      (1)
    \end{minipage}
    \hfill
    \begin{minipage}{0.4\linewidth}
      \begin{center}
        \includegraphics[scale=0.8]{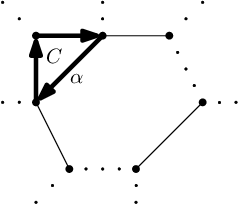}
      \end{center}
    \end{minipage}
    \hfill
    \begin{minipage}{0.05\linewidth}
      (2)
    \end{minipage}
    \hfill
    \begin{minipage}{0.4\linewidth}
      \begin{center}
        \includegraphics[scale=0.8]{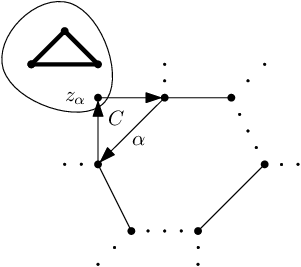}
      \end{center}
    \end{minipage}
  \end{minipage}

  \vspace{10pt}

  Similarly we define the parameters $s_1$ and $s_2$ with \lq
  anti-clockwise\rq\ replaced by \lq clockwise\rq.
\end{Def}

\begin{eg}\label{eg:quivers}
  We indicate the arrows which count for the parameter $r_1$
  by \includegraphics[scale=0.9]{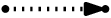} and the arrows which count for $s_1$
  by \includegraphics[scale=0.9]{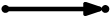}. Furthermore, the oriented $3$-cycles
  counting for $r_2$ are indicated by \includegraphics[scale=1]{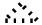}
  and the oriented $3$-cycles counting for $s_2$ are indicated by
  \includegraphics[scale=1]{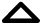}.
  
  \bigskip

  Consider the following realization of a quiver in $\mQ_{16}$:
    \begin{center}
      \includegraphics[scale=0.9]{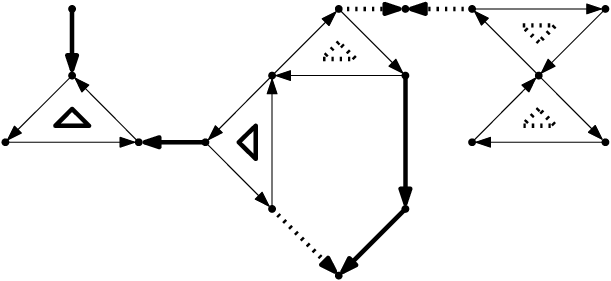}
    \end{center}
    Here, we have $r_1=3, \ r_2=3, \ s_1=4$ and $s_2=2$.
\end{eg}

In \cite{Bastian2009} an explicit description of the mutation classes of
quivers of type $\Atilde_{n-1}$ and, moreover, the derived equivalence
classes of cluster-tilted algebras of type $\Atilde_{n-1}$ is given as follows:

\begin{thm}\cite[Theorem 3.12, Theorem 5.5]{Bastian2009} \label{thm:description_mutation_classes} 
  Let $Q_1\in \mQ_{n-1}$ be a quiver with a realization having
  parameters $r_1$, $r_2$, $s_1$ and $s_2$ such that $r_1<s_1$ or
  $r_1=s_1$ and $r_2\leq s_2$.  Similarly, let $Q_2\in \mQ_{n-1}$ be
  a quiver with a realization having parameters $\tilde{r}_1$,
  $\tilde{r}_2$, $\tilde{s}_1$ and $\tilde{s}_2$ such that $\tilde
  r_1<\tilde s_1$ or $\tilde r_1=\tilde s_1$ and $\tilde
  r_2\leq\tilde s_2$.  Then $Q_1$ is mutation equivalent to $Q_2$ if
  and only if $r_1+2r_2 = \tilde{r}_1 + 2\tilde{r}_2$ and $s_1+2s_2 =
  \tilde{s}_1 + 2\tilde{s}_2$.

  
  Moreover, two cluster-tilted algebras of type $\Atilde_{n-1}$ are
  derived equivalent if and only if their quivers have realizations
  with the same parameters $r_1,$ $r_2,$ $s_1$ and $s_2$.
\end{thm}


\section{A Combinatorial Grammar}
\label{sec:comb-gramm}

In this section we describe the elements of the mutation classes of type
$\Atilde_{n-1}$ by a combinatorial grammar.  This can be viewed as an exercise in
the theory of species (introduced by Joyal, see the book~\cite{BLL} by
Bergeron, Labelle and Leroux) or the symbolic
method (as detailed in the recent book~\cite{FS} by Flajolet and Sedgewick).
We first give a recursive description of \emph{rooted quivers of type $A$}
as defined in \ref{dfn:description_mutation_class}.
A quiver of type $\Atilde_{n-1}$ will then be roughly a cycle of rooted quivers
of type $A$.

\subsection{A recursive description of rooted quivers of type $A$}
\label{sec:rooted-quivers-type}

Let $\Apoint$ be the set of all rooted quivers of type $A$.  We can then
describe the elements of $\Apoint$ recursively. A rooted quiver of
type $A$ is one of the following:
\begin{itemize}
\item the root;
\item the root, incident to an arrow, and a rooted quiver of type $A$ incident
  to the other end of the arrow.  The arrow may be directed either way.
\item the root, incident to an oriented $3$-cycle, and two rooted quivers of
  type $A$, each being incident to one of the other two vertices of the
  $3$-cycle.
\end{itemize}
We obtain the following combinatorial grammar:
\begin{center}
  \includegraphics[scale=1]{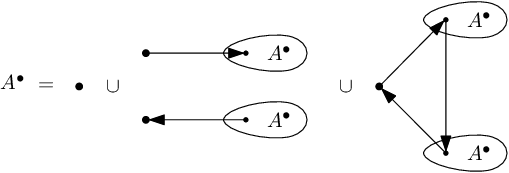}
\end{center}

We set the \Dfn{weight} of an arrow which is not part of an oriented
$3$-cycle equal to $z$ and the weight of an oriented $3$-cycle equal
to $tz^2$.  Hence, the weight of a rooted quiver $Q$ of type $A$ is
$z^{\#\{\text{vertices in }Q\}-1}t^{\#\{\text{oriented $3$-cycles in
  }Q\}}$.  This choice of weight is in accordance with the first part
of Theorem~\ref{thm:description_mutation_classes} where we count
oriented $3$-cycles in quivers (the number of which we denoted $r_2$,
resp.\ $s_2$) twice.

Thus, let
$$
\Apoint(z,t)=\sum_{Q\in\Apoint} z^{\#\{\text{vertices in $Q$}\}-1} t^{\#\{\text{oriented $3$-cycles in }Q\}}
$$ 
be the generating function (in particular: the formal power series) associated
to rooted quivers of type $A$.  From the recursive description, we obtain
\begin{gather*}
  \Apoint(z,t) = 1+2z\Apoint(z,t)+z^2t\Apoint(z,t)^2,%
  \intertext{or equivalently}%
  z^2t\Apoint(z,t)^2 + (2z-1)\Apoint(z,t)+ 1=0.
\end{gather*}
Solving this quadratic equation for $\Apoint(z,t)$ and choosing the
branch corresponding to a generating function gives
\begin{equation*}
\Apoint(z,t) = \frac{1-2z-\sqrt{1-4(z+(t-1)z^2)}}{2z^2t}.
\end{equation*}
We remark that for $t=1$ this is the generating function for the
\emph{Catalan numbers} shifted by $1$,
\begin{align*}
  \Apoint(z,1)& = \frac{1-2z-\sqrt{1-4z}}{2z^2}\\
  & = \sum_{n \geq 1} \frac{1}{n+1}\binom{2n}{n} z^{n-1},
\end{align*}
see e.g.~\cite[Section 3.0 Eq.~(3)]{BLL}.

To give a combinatorial description of the realizations of quivers in the
mutation classes of type $\Atilde_{n-1}$ corresponding to
Definition~\ref{dfn:parameters} we need auxiliary objects, which are one of the
following:
\begin{enumerate}
\item\label{enum:B1} a single (base) arrow, oriented from left to right, or
\item a rooted quiver of type $A$ attached to an oriented $3$-cycle, whose base
  arrow (see Definition~\ref{dfn:parameters}) is oriented from left to right,
  or
\item a single (base) arrow, oriented from right to left, or
\item\label{enum:B4} a rooted quiver of type $A$ attached to an oriented $3$-cycle, whose base
  arrow is oriented from right to left.
\end{enumerate}
\begin{rmk}
  The \lq base arrows\rq\ in \eqref{enum:B1}--\eqref{enum:B4} above will become
  precisely the arrows of the non-oriented cycle, which justifies the usage of
  the name.
\end{rmk}
Thus, we again obtain a combinatorial grammar:
\begin{center}
  \includegraphics[scale=1]{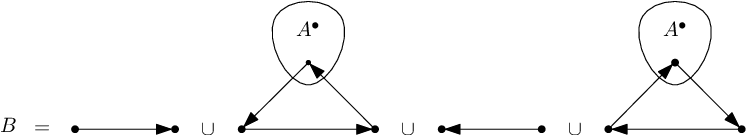}
\end{center}
\bigskip

The \Dfn{weight} of an object $Q \in B$ is $p^{\#\{\text{vertices in }Q\}-1}x^{\#\{\text{oriented $3$-cycles in }Q\}}$ if it is of
type (1) or (2), and $q^{\#\{\text{vertices in }Q\}-1}y^{\#\{\text{oriented $3$-cycles in }Q\}}$ if it is of type (3) or
(4). In particular, the weight of $Q$ depends only on the orientation of the
base arrow and on the total number of vertices and $3$-cycles of $Q$. Passing to generating
functions, we obtain
\begin{align}
  B(p,q,x,y)  &= p+p^2x\Apoint(p,x)+q+q^2y\Apoint(q,y) \nonumber \\
          &= \frac{1- \sqrt{1-4\big(p+(x-1)p^2\big)}}{2} + \frac{1- \sqrt{1-4\big(q+(y-1)q^2\big)}}{2} \nonumber \\
          &= \big(p+(x-1)p^2\big) C\big(p+(x-1)p^2\big) + \big(q+(y-1)q^2\big) C\big(q+(y-1)q^2\big), \nonumber 
\end{align}
where $C(z)$ is the generating function for the \Dfn{Catalan
  numbers},
$$
C(z) = \frac{1- \sqrt{1-4z}}{2z} = \sum_{n \geq 0} \frac{1}{n+1}\binom{2n}{n}
z^n.
$$
Note that
$$
B(p,q,x,y)=B\big(p+(x-1)p^2,q+(y-1)q^2,1,1\big)\;.
$$

\subsection{The number of quivers of type $\Atilde_{n-1}$}
\label{sec:number-quivers}

In this section we will first determine the number of realizations of
quivers of type $\Atilde_{n-1}$, as defined in
Definition~\ref{dfn:realization}.  This already suffices to determine
the number of quivers with parameters $r_1$, $r_2$, $s_1$, $s_2$ such
that $r_1<s_1$ or $r_1=s_1$ and $r_2<s_2$, see Corollary~\ref{cor:
  non-symm_case}.  We then count quivers with $r_1=s_1$ and $r_2=s_2$
that are \Dfn{symmetric}, i.e., whose two realizations coincide, to
determine the number of quivers in the general case as stated in
Corollary~\ref{cor:number-quivers-general}.

By Definition~\ref{dfn:description_mutation_class}, a realization of a quiver
of type $\Atilde_{n-1}$ is simply a cyclic arrangement of elements in $B$
with a total of $n$ vertices.  For example, the quiver in
Example~\ref{eg:quivers} consists of five elements of $B$, three of which
are just arrows, the two others are rooted quivers of type $A$ attached to an
oriented $3$-cycle.

The following Lemma is the so called \Dfn{cycle construction}, which is well
known in combinatorics, see eg.\ \cite[Eq.~(18), Section~1.4]{BLL} or
\cite[Theorem I.1, Section I.2.2]{FS}.
\begin{lem}
  Let $B(z)$ be the generating function for a family of unlabelled objects,
  where $z$ marks size.  Then the generating function for cycles of such
  objects is
  $$\sum_{k\geq 1}\frac{\phi(k)}{k}\log\left(\frac{1}{1-B(z^k)}\right),$$
  where $\phi(k)$ is Euler's totient function, i.e., the number of $1\leq d<k$
  coprime to $k$.
\end{lem}

Thus, we obtain for the generating function for realizations of quivers of type
$\Atilde_{n-1}$ with $p$ marking $r_1+2r_2$, $q$ marking $s_1+2s_2$, $x$ marking $r_2$
and $y$ marking $s_2$
$$
\Atilde(p, q, x, y)=\sum_{k\geq 1}
\frac{\phi(k)}{k}\log\left(\frac{1}{1-B(p^k,q^k,x^k,y^k)}\right).
$$

Let us first determine the coefficients in the special case of $\log\frac{1}{1-B(p,q,1,1)}$. 

\begin{lem}\label{lem:Gessel}
  For $(r,s)\neq(0,0)$ we have
  $$
  [p^rq^s]\log\left(\frac{1}{1-B(p,q,1,1)}\right)=
  \frac{1}{2(r+s)}\binom{2r}{r}\binom{2s}{s},
  $$
  where $[p^rq^s]G(p,q)$ denotes the coefficient of $p^rq^s$ in the
  formal power series $G(p,q)$.
\end{lem}

\begin{proof}
A direct calculation gives
\begin{multline}\label{eq:1}
  1+2t\frac d{dt}\log\left(\frac1{1-B(tp,tq,1,1)}\right)\\
  \begin{aligned}
    &=1+\frac{2t}{\sqrt{1-4tp}+\sqrt{1-4tq}}
    \left(\frac{2p}{\sqrt{1-4tp}}+\frac{2q}{\sqrt{1-4tq}}
    \right)\\
    &=\frac 1{\sqrt{1-4tp}+\sqrt{1-4tq}}
    \left(\sqrt{1-4tp}+\frac{4tp}{\sqrt{1-4tp}}+\sqrt{1-4tq}+\frac{4tq}{\sqrt{1-4tq}}
    \right)\\
    &=\frac 1{\sqrt{1-4tp}+\sqrt{1-4tq}}
    \left(\frac{1}{\sqrt{1-4tp}}+\frac{1}{\sqrt{1-4tq}}
    \right)\\
    &=\frac1{\sqrt{1-4tp}}\cdot\frac1{\sqrt{1-4tq}}\\
    &=\sum_{r,s\geq0}\binom{2r}r\binom{2s}sp^rq^st^{r+s}\;.
  \end{aligned}
\end{multline}
Denoting $a_{r,s}=[p^rq^s]\log\left(\frac{1}{1-B(p,q,1,1)}\right)$ we
have
\begin{align}\label{eq:2}
1+2t\frac d{dt}\log\left(\frac1{1-B(tp,tq,1,1)}\right)%
&=1+2t\frac d{dt}\sum_{r,s\geq0}a_{r,s}p^rq^st^{r+s}\\\notag%
&=1+2\sum_{r,s\geq0}(r+s)a_{r,s}p^rq^st^{r+s}.
\end{align}
We now obtain $a_{r,s}$ by equating coefficients on the right hand
sides of Equation~\eqref{eq:1} and Equation~\eqref{eq:2}.
\end{proof}

We can now determine the coefficients of $\log\frac{1}{1-B(p,q,x,y)}$.

\begin{lem}\label{lem:Thomas}
\begin{multline*}
  [p^rq^sx^{r_2}y^{s_2}]\log\left(\frac{1}{1-B(p,q,x,y)}\right)\\
  =(-1)^{r+r_2+s+s_2}%
  \sum_{\substack{i,j\geq0\\(i,j)\neq(0,0)}}%
  \frac{(-1)^{i+j}}{2(i+j)}%
  \binom{2i}{i,2i-r,r_2,r-r_2-i}\binom{2j}{j,2j-s,s_2,s-s_2-j},
\end{multline*}
where 
$[p^rq^sx^{r_2}y^{s_2}]B(p,q,x,y)$ denotes the coefficient of
  $p^rq^sx^{r_2}y^{s_2}$ in the formal power series $B(p,q,x,y)$.
\end{lem}

\begin{proof}
From Lemma \ref{lem:Gessel} and the substitution $B(p,q,x,y)=B(p+(x-1)p^2,q+(y-1)q^2,1,1)$ 
it follows that
$$
\log\left(\frac1{1-B(p,q,x,y)}\right)=
\sum_{\substack{i,j\geq0\\(i,j)\neq(0,0)}}\frac1{2(i+j)}{2i\choose i}{2j\choose j}p^i(1+(x-1)p)^iq^j(1+(y-1)q)^j\;.
$$
A simple expansion gives now
\begin{multline*}
  \log\left(\frac1{1-B(p,q,x,y)}\right)\\
  \shoveright{%
    \begin{aligned}
      &=\sum_{\substack{i,j\geq0\\(i,j)\neq(0,0)}}\sum_{k,l,r_2,s_2\geq0}%
      p^{i+k}q^{j+l}x^{r_2}y^{s_2}\frac{(-1)^{k+r_2+l+s_2}}{2(i+j)}%
      {2i\choose i}{2j\choose j}%
      {i\choose k}{j\choose l}{k\choose r_2}{l\choose s_2}\\
      &=\sum_{r,s,r_2,s_2\geq0}p^rq^sx^{r_2}y^{s_2}\\
    \end{aligned}}\\
  \sum_{\substack{i,j\geq0\\(i,j)\neq(0,0)}}%
  \frac{(-1)^{r+s+r_2+s_2+i+j}}{2(i+j)}%
  {2i\choose i}{2j\choose j}%
  {i\choose r-i}{j\choose s-j}{r-i\choose r_2}{s-j\choose s_2},
\end{multline*}
from which one reads off the desired result.
\end{proof}
      
Putting the pieces together we obtain:

\begin{thm} \label{th:realization of quivers} The number of
  realizations of quivers of type $\Atilde_{r+s-1}$ with parameters
  $r>0$ and $s>0$ is given by
  \begin{equation}
    \label{eq:realizations-r-s}
    \frac{1}{2}
    \sum_{k|r,k|s}\frac{\phi(k)}{r+s}\binom{2r/k}{r/k}\binom{2s/k}{s/k}.
  \end{equation}
  The number of realizations of quivers of type $\Atilde_{r+s-1}$
  with parameters $r_1, r_2, s_1, s_2$ such that $r=r_1+2r_2>0$ and
  $s=s_1+2s_2>0$ is given by
  \begin{multline} 
    \label{eq:realizations-r1-s1-r2-s2}
    \sum_{k|r,k|r_2,k|s,k|s_2}\frac{\phi(k)}k(-1)^{(r+r_2+s+s_2)/k}\\
    \sum_{\substack{i,j\geq0\\(i,j)\neq(0,0)}}\frac{(-1)^{i+j}}{2(i+j)}
    \binom{2i}{i,2i-r/k,r_2/k,(r-r_2)/k-i}\binom{2j}{j,2j-s/k,s_2/k,(s-s_2)/k-j}.
  \end{multline}
\end{thm}
\begin{proof}
  Observe that for any $F(p,q)=\sum_{r,s} f_{r,s} p^r q^s$ we have
  $$
  [p^r q^s]F(p^k,q^k)=
  \begin{cases}
    f_{r/k,s/k}&\text{when $k|r$ and $k|s$,}\\
    0          &\text{otherwise}.
  \end{cases}
  $$
  Using Lemma~\ref{lem:Gessel} we get
  \begin{align*}
    [p^r q^s]\sum_{k\geq 1}\frac{\phi(k)}{k}
                           \log\left(\frac{1}{1-B(p^k,q^k,1,1)}\right)
    &=\sum_{k\geq 1}\frac{\phi(k)}{k}
                    [p^r q^s]\log\left(\frac{1}{1-B(p^k,q^k,1,1)}\right)\\
    &=\sum_{k|r, k|s}\frac{\phi(k)}{k}
                    \frac{k}{2(r+s)}\binom{2r/k}{r/k}\binom{2s/k}{s/k}.
  \end{align*}

  The general formula follows similarly from Lemma~\ref{lem:Thomas}.
\end{proof}

As a corollary we obtain the number of quivers of type
$\Atilde_{r+s-1}$ with parameters that do not coincide:
\begin{cor}\label{cor: non-symm_case}
  For $r<s$, the number $\tilde a(r,s)$ of quivers of type
  $\Atilde_{r+s-1}$ with parameters $r$ and $s$ is given by
  Formula~\eqref{eq:realizations-r-s}.  For $r_1<s_1$ or $r_1=s_1$
  and $r_2<s_2$, the number of quivers with parameters $r_1$, $r_2$,
  $s_1$ and $s_2$ is given by
  Formula~\eqref{eq:realizations-r1-s1-r2-s2}.
\end{cor}

\begin{proof}
  If $r_1<s_1$ or $r_1=s_1$ and $r_2<s_2$, a quiver has a unique
  realization with these parameters.  Therefore, the claim follows
  directly from Theorem~\ref{th:realization of quivers}.
\end{proof}

We have seen that a quiver of type $\Atilde_{2r-1}$ is a non-oriented cycle of
elements in $B$ with a total number of $2r$ vertices.  To count quivers of type
$\Atilde_{2r-1}$, we first have to consider \Dfn{symmetric quivers} of type
$\Atilde_{2r-1}$, i.e., quivers where both possible realizations coincide.  To
do so, we have to count lists of elements in $B$:
\begin{lem}
  The number of lists $(B_1, \dots, B_\ell)$ of elements in $B$ with
  a total of $r+\ell$ vertices is given by the central binomial
  coefficient $\binom{2r}{r}$.  The number of such lists with $r_2$
  oriented $3$-cycles is given by
  \begin{equation}\label{eq:lists}
    2^{r-2r_2}\binom{r}{r_2,r_2,r-2r_2}=2^{r-2r_2}\binom{r}{2r_2}\binom{2r_2}{r_2}.
  \end{equation}
\end{lem}
\begin{proof}
  The generating function for elements in $B$ taking into account
  only the number of vertices is $B(p,p,1,1) = 1 - \sqrt{1-4p}$.
  Thus, we obtain that the number of lists of elements in $B$ with
  $r+\ell$ vertices in total is given by
  $$
  [p^r] \frac{1}{1-B(p,p,1,1)} = [p^r] \frac{1}{\sqrt{1-4p}} = [p^r] \sum_{n \geq
    0} \binom{2n}{n} p^n = \binom{2r}{r},
  $$
  compare \cite[Example 1.2.2(a) and Theorem 1.4.2]{BLL}.

  Let us now prove the more refined statement, by giving a meaning to
  each of the factors in the last expression of
  Equation~\eqref{eq:lists}.  We first observe that $r_1=r-2r_2$ is
  precisely the number of arrows that are not part of an oriented
  $3$-cycle, and thus $2^{r-2r_2}$ is the number of their possible
  orientations.

  The central binomial coefficient $\binom{2r_2}{r_2}$ can be
  interpreted as the number of lists $L^{B_\Delta}=(B_1, \dots,
  B_\ell)$ of elements in $B$, where all elements consist of oriented
  $3$-cycles only: namely, such a list is either empty, or its first
  element is an oriented $3$-cycle (with its two possible
  orientations), to which a rooted quiver of type $A$, consisting of
  oriented $3$ cycles only, is attached.  It is easy to see that the
  generating function for such rooted quivers is
  $\frac{1-\sqrt{1-4x}}{2x}$.  Let us denote the generating function
  for the lists under consideration $L^{B_\Delta}(x)$.  We then have:
  \begin{align*}
    L^{B_\Delta}(x)
    &=1+2x\frac{1-\sqrt{1-4x}}{2x} L^{B_\Delta}(x)\\
    &=1+(1-\sqrt{1-4x}) L^{B_\Delta}(x)\\
    &=\frac{1}{\sqrt{1-4x}}.
  \end{align*}

  Finally, $\binom{r}{2r_2}=\binom{(2r_2+1) + r_1 -1}{r_1}$ is the
  number of ways to choose $r_1$ vertices (with repetitions) in a
  list $L^{B_\Delta}$ where arrows can be inserted to obtain a list
  of elements in $B$ with $r+\ell$ vertices and $r_2$ oriented
  $3$-cycles.  Namely, there are $2r_2+\ell$ vertices in total in
  $L^{B_\Delta}$, all but the $\ell-1$ vertices which are at the left
  of the base-arrows in $B_2, \dots, B_\ell$ are possible insertion
  places.
\end{proof}

Given a list $L = (B_1,\dots,B_\ell)$ of elements in $B$, we identify $L$ with
the quiver obtained from $L$ by gluing together the right vertex in the base
arrow of $B_i$ and the left vertex in the base arrow of $B_{i+1}$ for $1 \leq i
< \ell$. 
For a list $L = (B_1,\dots,B_\ell)$ of elements in $B$ define the
\Dfn{reversed list} $\rev(L) := (\overline B_\ell,\dots,\overline B_1)$, where
$\overline B_i$ is obtained from $B_i$ by reversing the direction of the base
arrow of $B_i$ (and eventually of the associated oriented $3$-cycle). See
Figures~\ref{eg:symmetric quivers}$(b)$ and \ref{eg:symmetric quivers}$(c)$
for an example. 
Obviously, we have $\rev(\rev(L)) = L$.

\begin{figure}
  \begin{tabular}{ccc}
    \begin{tabular}{c}
      \includegraphics[scale=0.7]{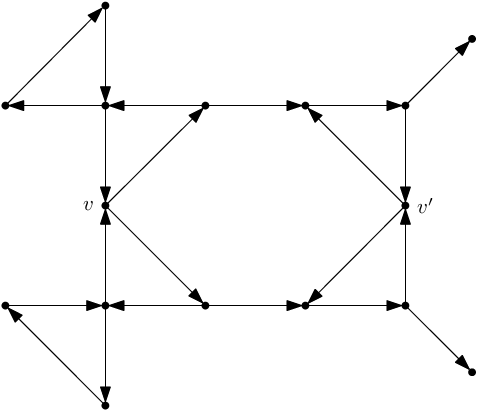} \\ $(a)$
    \end{tabular}
    \begin{tabular}{c}
      \includegraphics[scale=0.7]{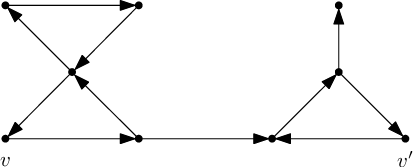} \\ $(b)$ \\[15pt]
      \includegraphics[scale=0.7]{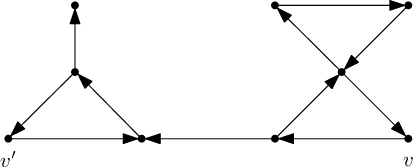} \\ $(c)$
    \end{tabular} & \hspace{7pt} &
  \end{tabular}
  \caption{$(a)$ a symmetric quiver of type $\Atilde_{15}$; $(b)$ the list $L$
    of elements in $B$ starting at $v$ end ending at $v'$; $(c)$ the list
    $\rev(L)$ of elements in $B$ starting at $v'$ end ending at $v$.}
  \label{eg:symmetric quivers}
\end{figure}

\begin{thm} \label{th:symmetric quivers} The number of symmetric
  quivers of type $\Atilde_{2r-1}$, i.e., quivers where both possible
  realizations coincide, is equal to $\frac{1}{2}\binom{2r}{r}$.  The
  number of symmetric quivers of type $\Atilde_{2r-1}$ with $2r_2$
  oriented $3$-cycles is
  \begin{equation*}
    2^{r-2r_2-1}\binom{r}{r_2,r_2,r-2r_2}
  \end{equation*}
\end{thm}
\begin{proof}
  Starting with a list $L$ of elements in $B$ with a total of $r+\ell$ vertices,
  we obtain a symmetric quiver of type $\Atilde_{2r-1}$ by taking $L$ and
  $\rev(L)$, and gluing together the end point of $L$ with the start point of
  $\rev(L)$ and vice versa. E.g., the symmetric quiver in
  Figure~\ref{eg:symmetric quivers}$(a)$ is obtained from the lists $L$ and
  $\rev(L)$ shown in Figures~\ref{eg:symmetric quivers}$(b)$ and
  \ref{eg:symmetric quivers}$(c)$.

  To prove the statement it remains to show that exactly two different lists
  belong to the given symmetric quiver $Q$. Observe first, that $Q$ is of the
  form $Q=(L,L')$ where the end point of $L$ is glued together with the start
  point of $L'$ and vice versa, such that furthermore, $L' = \rev(L)$ is the
  reversed list of $L$. It may happen that $L$ is itself symmetric, i.e., $L =
  \rev(L)$. However, it is always possible to find a non-symmetric $X$ such
  that $Y := \rev(X) \neq X$ and $L=(X,Y,X,\dots,Y)$ and $L'=(\overline
  Y,\overline X,\overline Y,\dots, \overline X)$. That is, any symmetric quiver
  is of the following form:
  \begin{center}
    \includegraphics[scale=0.8]{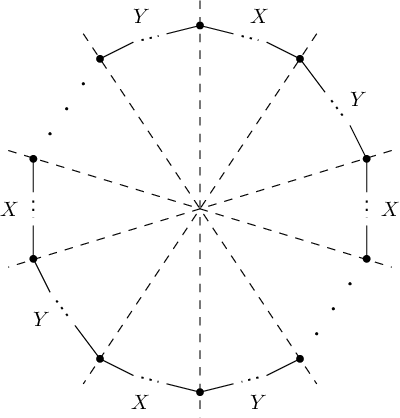}
  \end{center}

  This proves that there exist exactly two different lists that correspond to a
  symmetric quiver $Q$, namely $L$ and $L'$.
\end{proof}

We now know the number of \emph{realizations} of quivers as well as
the number of \emph{symmetric} quivers of type $\Atilde_{2r-1}$ with
parameters $r$ and $s=r$.  Therefore, we can also compute the total
number of quivers of type $\Atilde_{2r-1}$ with the same parameters:

\begin{cor}\label{cor:number-quivers-general}
  The number $\tilde a(r,r)$ of quivers of type $\Atilde_{2r-1}$ with
  parameters $r$ and $s=r$ is given by
  $$
  \frac{1}{2} \left(\frac{1}{2} \binom{2r}{r} + \sum_{k|r} \frac{\phi(k)}{4r}
  \binom{2r/k}{r/k}^2 \right).
  $$

  The number of quivers of type $\Atilde_{2r-1}$ with parameters
  $r_1$, $r_2$, $s_1$ and $s_2$ such that $r_1=s_1$ and $r_2=s_2$ is
  given by
  \begin{multline*} 
    2^{r-2r_2-2}\binom{r}{r_2,r_2,r-2r_2}\\ %
      +\sum_{\substack{k|r,k|r_2\\i,j\geq0\\(i,j)\neq(0,0)}}\frac{\phi(k)}k
    \frac{(-1)^{i+j}}{4(i+j)}
    \binom{2i}{i,2i-r/k,r_2/k,(r-r_2)/k-i}\binom{2j}{j,2j-r/k,r_2/k,(r-r_2)/k-j},
  \end{multline*}
  where $r=r_1+2r_2$.
\end{cor}
\begin{proof}
  According to Theorem~\ref{th:realization of quivers}, the
  expression $\sum_{k|r} \frac{\phi(k)}{4r} \binom{2r/k}{r/k}^2$
  counts realizations of quivers with parameters $r$ and $s=r$.
  Therefore, it counts non-symmetric quivers with parameters $r$ and
  $s=r$ twice and symmetric quivers with parameters $r$ and $s=r$
  once.  By Theorem~\ref{th:symmetric quivers}, the number of
  symmetric quivers with parameters $r$ and $s=r$ is given by
  $\frac{1}{2} \binom{2r}{r}$.  In total, we get the desired
  expression.  The general case is dealt with similarly.
\end{proof}

\begin{table}
  \centering
  \begin{tabular}{|c|rrrrr|}
    \hline
    2 &    1 &&&&\\
    3 &    2 &&&&\\
    4 &    5 &    4&&&\\
    5 &   14 &   12&&&\\
    6 &   42 &   36 &   22&&\\
    7 &  132 &  108 &  100&&\\
    8 &  429 &  349 &  315 &  172&\\
    9 & 1430 & 1144 & 1028 &  980&\\
   10 & 4862 & 3868 & 3432 & 3240 & 1651\\\hline
   \slashbox{n}{r}& 1   & 2    & 3    & 4    & 5 \\
   \hline
  \end{tabular}
  \bigskip
  \caption{Number of quivers of type $\Atilde_{n-1}$ according to the
    parameter $r$ for $n$ in $\{2,3,\dots,10\}$}
  \label{tab:numbers}
\end{table}

\subsection{The number of quivers of type $D_n$}

With the help of Corollary~\ref{cor: non-symm_case} and a little extra
work we obtain the number of quivers in the mutation class of Dynkin
type $D_n$. This result was first determined by Buan and Torkildsen in
\cite{BuanTorkildsen2009}.

\begin{cor} The number of quivers of type $D_n$, for $n \geq 5$, is given by
  $$\tilde{a}(0,n)=\sum_{d|n} \frac{\phi(n/d)}{2n} \binom{2d}{d}.$$
  The number of quivers of type $D_4$ is $6$.
\end{cor}

\begin{proof}
  For $n = 4$, the quivers can be explicitly listed, see \cite{BuanTorkildsen2009}.
  We remark that their number does not agree with the general formula.
  Now, let $\bar D_n$, $n \geq 5$, be the family of cyclic arrangements of elements
  in $B$, with all base arrows oriented clockwise and a total of $n$ vertices.
  Thus, the elements in $\bar D_n$ are quivers with a distinguished oriented cycle,
  which we call the \Dfn{main cycle}.  Note that the main cycle may be an oriented
  $2$-cycle or even a loop.

  We want to show that the quivers of type $D_n$ are in bijection with those in
  $\bar D_n$.  To do so, we use the classification given by Vatne~\cite{Vatne2008},
  who distinguishes four types $I$--$IV$.  Quivers in
  $D_n$ of type $IV$ coincide with those objects in $\bar D_n$ whose main cycle
  consists of at least three arrows.  The other three types are as in
  Figure~\ref{fig:Dn-Quivers}.

  \begin{figure}[ht]
    \centering
    \begin{tabular}{c@{\hspace{1cm}}c}
    \includegraphics[scale=0.7]{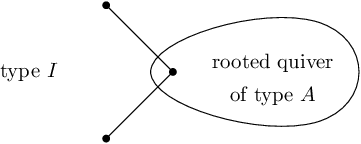} & \includegraphics[scale=0.7]{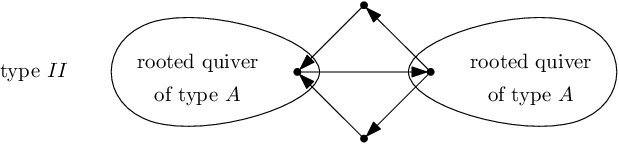}
    \end{tabular}
    
    \vspace{0.5cm}
    \centering
	\includegraphics[scale=0.7]{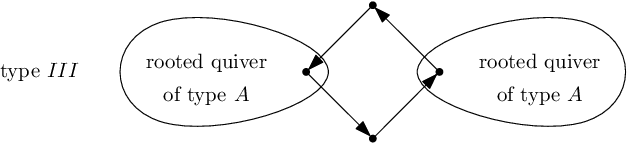}
    \caption{Quivers in $D_n$ of type $I$--$III$.}
    \label{fig:Dn-Quivers}
  \end{figure}

  Suppose that the main cycle of $\bar Q\in\bar D_n$ is an oriented $2$-cycle.
  By deleting these two arrows we obtain one of the following:
  \begin{enumerate}
  \item a quiver in $D_n$ of type $I$, where precisely one of the two
    distinguished arrows incident to the root is oriented towards it, or
  \item a quiver in $D_n$ of type $III$, i.e., a quiver having a unique
    oriented $4$-cycle.
  \end{enumerate}

  It remains to describe the bijection in the case where the main cycle of
  $\bar Q\in\bar D_n$ is a loop.  In a first step, we delete the vertex of this
  loop and all arrows incident to it, to obtain a rooted quiver $\bar
  Q^\bullet$ of type $A$.  For the second and final step, we distinguish two
  cases:
  \begin{enumerate}
  \item the root of $\bar Q^\bullet$ is incident to a single arrow $\alpha$.
    In this case we obtain a quiver $Q$ in $D_n$ of type $I$ by adding a second
    arrow, oriented in the same way as $\alpha$, to the other vertex $\alpha$
    is incident to.
  \item On the other hand, consider the case that the root of $\bar Q^\bullet$
    is incident to an oriented $3$-cycle $\gamma$.  Then, we glue a second
    $3$-cycle, oriented in the same way as $\gamma$, along the arrow of
    $\gamma$ opposite to the root.  In this way we create a quiver in $D_n$
    of type $II$.
  \end{enumerate}

  This transformation is invertible: 
  \begin{itemize}
  \item a quiver $Q$ in $D_n$ of type $I$ has a uniquely determined root,
    and two distinguished arrows incident to it.  If they are oriented in
    opposite directions, then the main cycle in the preimage of the
    transformation is an oriented $2$-cycle.  Otherwise, the preimage is a
    loop.
  \item $Q$ is of type $II$, if and only if it has two oriented $3$-cycles
    sharing an arrow.
  \item Finally, $Q$ is of type $III$, if and only if it has a unique oriented
    $4$-cycle.
  \end{itemize}

  To conclude, we compute the number of elements in $\bar D_n$.  This is easy,
  since we can use the degenerate case of $r=0$ and $s=n$ of
  Corollary~\ref{cor: non-symm_case}:
  \begin{eqnarray*}
    \tilde a(0,n)
    &=& \frac{1}{2} \sum_{k|n}\frac{\phi(k)}{n}\binom{2n/k}{n/k}\\
    &=& \frac{1}{2} \sum_{d|n}\frac{\phi(n/d)}{n}\binom{2d}{d},\quad\text{for $d:=\frac{n}{k}$}.
  \end{eqnarray*}
\end{proof}

\providecommand{\cocoa} {\mbox{\rm C\kern-.13em o\kern-.07em C\kern-.13em
  o\kern-.15em A}}
\providecommand{\bysame}{\leavevmode\hbox to3em{\hrulefill}\thinspace}
\providecommand{\href}[2]{#2}


\begin{thebibliography}{10}

\bibitem{AssemBruestleCharbonneauJodoinPlamondon2009}
Ibrahim Assem, Thomas Br\"ustle, Gabrielle Charbonneau-Jodoin, Pierre-Guy Plamondon,
  \emph{Gentle algebras arising from surface triangulations},
  Algebra \& Number Theory \textbf{4} (2010), no.~2, 201--229, \mbox{math.RT/0903.3347}.

\bibitem{Bastian2009}
Janine Bastian, \emph{{Mutation classes of $\tilde{A}_n-$quivers and derived
  equivalence classification of cluster tilted algebras of type
  $\tilde{A}_n$}}, to appear in Algebra \& Number Theory (2010), 24 pp.,
  \mbox{math.RT/0901.1515}.

\bibitem{BLL}
Fran{\c{c}}ois Bergeron, Gilbert Labelle, and Pierre Leroux,
  \emph{Combinatorial species and tree-like structures}, Encyclopedia of
  Mathematics and its Applications, vol.~67, Cambridge University Press,
  Cambridge, 1998, Translated from the 1994 French original by Margaret Readdy,
  With a foreword by Gian-Carlo Rota.

\bibitem{BuanMarshReiten2008} Aslak~Bakke Buan, Robert Marsh and Idun
  Reiten, \emph{{Cluster mutation via quiver representations}},
  Commentarii Mathematici Helvetici \textbf{83} (2008), no.~1,
  143--177, \mbox{math.RT/0412077}.


\bibitem{BuanTorkildsen2009}
Aslak~Bakke Buan and Hermund~Andr\'e Torkildsen, \emph{{The Number of Elements
  in the Mutation Class of a Quiver of Type $D_n$}}, Electronic Journal of
  Combinatorics \textbf{16} (2009), no.~1, Research Paper 49, 23 pp.
  (electronic), \mbox{math.RT/0812.2240}.

\bibitem{BuanVatne2008}
Aslak~Bakke Buan and Dagfinn~F. Vatne, \emph{Derived equivalence classification
  for cluster-tilted algebras of type {$A\sb n$}}, Journal of Algebra
  \textbf{319} (2008), no.~7, 2723--2738, \mbox{math.RT/0701612}.

\bibitem{MR2187656} Philippe Caldero, Fr\'ed\'eric Chapoton, and Ralf
  Schiffler, \emph{Quivers with relations arising from clusters ({$A_n$}
              case)}, Transactions of the American Mathematical Society
  \textbf{358} (2006), no.~3, 1347--1364, \mbox{math.RT/0401316}.

\bibitem{FS}
Philippe Flajolet and Robert Sedgewick, \emph{Analytic combinatorics},
  Cambridge University Press, Cambridge, 2009.

\bibitem{FominShapiroThurston2008}
Sergey Fomin, Michael Shapiro, and Dylan Thurston, \emph{Cluster algebras and
  triangulated surfaces. {I}. {C}luster complexes}, Acta Mathematica
  \textbf{201} (2008), no.~1, 83--146, \mbox{math.RA/0608367}.

\bibitem{FominZelevinsky2002}
Sergey Fomin and Andrei Zelevinsky, \emph{Cluster algebras. {I}.
  {F}oundations}, Journal of the American Mathematical Society \textbf{15}
  (2002), no.~2, 497--529 (electronic), \mbox{math.RT/0104151}.

\bibitem{Torkildsen2008}
Hermund~Andr\'e Torkildsen, \emph{{Counting cluster-tilted algebras of type
  $A_n$}}, International Electronic Journal of Algebra \textbf{4} (2008),
  149--158, \mbox{math.RT/0801.3762}.

\bibitem{Vatne2008}
Dagfinn~F. Vatne, \emph{{The mutation class of $D_n$ quivers}},
Comm. Algebra \textbf{38} (2010), no.~3, 1137--1146, \mbox{math.CO/0810.4789}.

\end{thebibliography}
\end{document}